\documentclass[preprint]{elsarticle}
    \makeatletter
    \def\ps@pprintTitle{%
      \let\@oddhead\@empty
      \let\@evenhead\@empty
      \def\@oddfoot{\reset@font\hfil\thepage\hfil}
      \let\@evenfoot\@oddfoot
    }
    \makeatother%
\usepackage{amsmath}%
\usepackage{amsfonts}%
\usepackage{amssymb}%
\usepackage{graphicx}
\usepackage{tikz}
\usetikzlibrary{plotmarks}

\newtheorem{theorem}{Theorem}

\newtheorem{corollary}{Corollary}

\newtheorem{proposition}{Proposition}

\newproof{proof}{Proof}

\begin{document}

\title{Decomposing Weighted Graphs}
\author{Amir Ban}
\ead{amirban@netvision.net.il}
\address{Blavatnik School of Computer Science, Tel Aviv University, Tel Aviv, Israel}
\date{}

\begin{abstract}
We solve the following problem: Can an undirected weighted graph $G$ be partitioned into two non-empty induced subgraphs satisfying minimum constraints for the sum of edge weights at vertices of each subgraph? We show that this is possible for all constraints $a(x), b(x)$ satisfying $d_G(x) \geq a(x) + b(x) + 2W_G(x)$, for every vertex $x$, where $d_G(x), W_G(x)$ are, respectively, the sum and maximum of incident edge weights.
\end{abstract}

\maketitle

\section{Introduction}

All graphs considered in this paper are finite, undirected and weighted. A weighted graph is a triple $G = (V, E, w)$ such that  $(V, E)$ is an undirected simple finite graph and $w : E \mapsto \mathbb{R}_{>0}$ is a weight function. Where $xy \notin E$, we further define $w_{xy} = w_{yx} = 0$.

We denote by $V(G)$ the vertex set of a graph $G$. The degree of vertex $x$ with respect to $G$ is denoted by $d_G(x)$ and is the sum of its incident edge weights: $d_G(x) = \sum\limits_{y \in V(G)} w_{xy}$. If $G(X)$ is the subgraph induced by $X$, we use the notation $d_X(x)$ as shorthand for $d_{G(X)}(x)$.

We denote by $W_G(x)$ the maximum weight (not including loop edges) of an edge of $G$ incident to $x$: $W_G(x) := \max\limits_{y \in V(G), x \ne y} w_{xy}$. $W(x)$ stands for $W_G(x)$ when the context is clear.

$(A, B)$ is called a {\em partition} of a set $V$ if $A, B$ are disjoint, non-empty subsets of $V$ whose union is $V$.

Stiebitz \cite{Stiebitz} proved the following decomposition result for a simple, undirected graph $G$: Let $a, b: V(G) \mapsto \mathbb{N}$ be two functions, and assume $d_G(x) \geq a(x) + b(x) + 1$ for every $x \in V(G)$. Then there is a partition $(A,B)$ of $V(G)$ such that
\begin{enumerate}[(1)]
\item $d_A(x) \geq a(x)$ for every $x \in A$, and
\item $d_B(x) \geq b(x)$ for every $x \in B$.
\end{enumerate}

Stiebitz' result does not lend itself to a natural generalization to weighted graphs, because of the restriction to integers on the vertex functions $a, b$. If this is relaxed, the theorem breaks. The Stiebitz proof in several places relies on vertex degrees being integers.

Nonetheless, in this paper we generalize this result to undirected weighted graphs, and base our proof closely on Stiebitz' proof. We prove the following result:

\begin{theorem}
\label{weighted}
Let $G$ be a graph without loop edges, and $a, b : V(G) \mapsto \mathbb{R}_{\geq 0}$ two functions. Assume that $d_G(x) \geq a(x) + b(x) + 2W_G(x)$ for every vertex $x \in V(G)$. Then there is a partition $(A,B)$ of $V(G)$ such that
\begin{enumerate}[(1)]
\item $d_A(x) \geq a(x)$ for every $x \in A$, and
\item $d_B(x) \geq b(x)$ for every $x \in B$.
\end{enumerate}

\end{theorem}

\section{Proof of Theorem \ref{weighted}}

Let $G$ be a graph and $a, b : V(G) \mapsto \mathbb{R}_{\geq 0}$ two functions such that
\begin{align*}
d_G(x) \geq a(x) + b(x) + 2W(x)
\end{align*}

\noindent for every $x \in V(G)$.

Let $f : V(G) \mapsto \mathbb{R}_{\geq 0}$ be a function. $G$ is said to be {\em $f$-meager} if for every induced subgraph $H$ of $G$ there is a vertex $x \in V(H)$ such that $d_H(x) < f(x) + W(x)$.\footnote{Stiebitz uses $d_H(x) \leq f(x)$ and names it {\em $f$-degenerate}, which is the standard terminology for this simple-graph property. The generalization used here is non-standard and non-obvious, and therefore we call it differently.}

We say a pair $(A,B)$ is {\em stable} if $A$ and $B$ are disjoint, non-empty subsets of $V(G)$ such that
\begin{enumerate}[(1)]
\item $d_A(x) \geq a(x)$ for every $x \in A$, and
\item $d_B(x) \geq b(x)$ for every $x \in B$.
\end{enumerate}

We have to show that there is a stable partition of $V(G)$. Following \cite{Stiebitz}, we make the following observation.

\begin{proposition}
\label{complete}
If there exists a stable pair, then there exists a stable partition of $V(G)$, too.
\end{proposition}

\begin{proof}
Let $(A,B)$ be a stable pair such that $A \cup B$ is maximal. We need only to show that $A \cup B = V(G)$. Suppose not, i.e. $C = V(G) \setminus (A \cup B)$ is non-empty. Then the maximality of $A \cup B$ implies that $(A,B \cup C)$ is not stable. Therefore, there is a vertex $x \in C$ such that $d_{B \cup C}(x) < b(x)$. Since $d_G(x) \geq a(x) + b(x) + 2W(x)$, $d_A(x) > a(x) + 2W(x) \geq a(x)$. But then $(A \cup \{x\},B)$ is a stable pair, contradicting the maximality of $A \cup B$. This proves the proposition.
\qed
\end{proof}

We define a {\em meager partition} of $V(G)$ as a partition $(A,B)$ of $V(G)$ such that $G(A)$ is $a$-meager and $G(B)$ is $b$-meager.

We define a function $h(A,B)$ of a partition $(A,B)$ as
\begin{align*}
h(A,B) = \sum\limits_{x \in A,y \in A} w_{xy} + \sum\limits_{x \in B,y \in B} w_{xy} + \sum\limits_{x \in A} b(x) + \sum\limits_{x \in B} a(x)
\end{align*}

For the proof of Theorem \ref{weighted} we consider two possible cases.

\begin{itemize}
\item There is no meager partition of $V(G)$. Then, among all non-empty subsets of $V(G)$ select one, say $A$, such that
\begin{enumerate}[(i)]
\item \label{first} $d_A(x) \geq a(x)$ for all $x \in A$, and
\item \label{second} $|A|$ is minimum subject to (\ref{first})
\end{enumerate}

Let $B = V(G) \setminus A$. Since $V(G) \setminus \{v\}$ satisfies (\ref{first}) for each vertex $v$, $A$ exists and is a proper subset of $V(G)$. Hence $B$ is non-empty. Because of (\ref{second}), for every non-empty proper subset $A'$ of $A$ there is a vertex $x \in A'$ such that $d_{A'}(x) < a(x)$. This implies that $d_A(x) < a(x) + W(x)$ for some $x \in A$. Hence $G(A)$ is $a$-meager. Clearly $G(B)$ is not $b$-meager, since otherwise $(A,B)$ would be a meager partition of $V(G)$. Therefore, there is a non-empty subset $B'$ of $B$ such that $d_{B'}(x) \geq b(x) + W(x) \geq b(x)$ for all $x \in B'$. Then $(A,B')$ is a stable pair and, by Proposition \ref{complete}, there is a stable partition of $V(G)$.

\item There is a meager partition of $V(G)$. Then let $(A,B)$ be a meager partition of $V(G)$ such that $h(A,B)$ is maximum. $G(A)$ being $a$-meager, there is a vertex $x \in A$ such that $d_A(x) < a(x) + W(x)$. Since $d_G(x) \geq a(x) + b(x) + 2W(x)$, $d_B(x) > b(x) + W(x)$. This implies that $|B| \geq 2$. By symmetry we also have $|A| \geq 2$.

Next, we claim that there is a non-empty subset $\bar{A} \subseteq A$ such that $d_{\bar{A}}(x) \geq a(x)$ for all $x \in \bar{A}$. Suppose not. Then, clearly, for each $y \in B$, $G(A \cup \{y\})$ is $a$-meager. $G(B)$ being $b$-meager, there is a vertex $y' \in B$ such that $d_B(y') < b(y') + W(y')$. Let $A' = A \cup \{y'\}$ and $B' = B \setminus \{y'\}$. Obviously, $B'$ is non-empty. Now, we easily conclude that $(A',B')$ is a meager partition of $V(G)$. Since $d_G(y') \geq a(y') + b(y') + 2W(y')$ and $d_B(y') < b(y') + W(y')$, we have $d_{A'}(y') > a(y') + W(y')$ and, therefore,
\begin{align*}
h(A',B') - h(A,B) = d_{A'}(y') - d_B(y') + b(y') - a(y')  > 0
\end{align*}

\noindent contradicting the maximality of $h(A,B)$. This proves the claim. By symmetry there is a non-empty subset $\bar{B} \subseteq B$ such that $d_{\bar{B}} \geq b(x)$ for all $x \in \bar{B}$. Then $(\bar{A},\bar{B})$ is a stable pair, and, by Proposition \ref{complete}, there is a stable partition of $V(G)$.

This completes the proof of Theorem \ref{weighted}.
\end{itemize}

\section{Concluding Remarks}

Theorem \ref{weighted} is tight in view of graphs where $w_{xy} = 1$ for all $x \neq y$: E.g. $K_9$ with unit edge weights has degree $8$ for every vertex. Setting $a(x) = b(x) = 3 + \epsilon$, no stable partition exists for any $\epsilon > 0$.

We can generalize Theorem \ref{weighted} to the case of a weighted undirected graph with loops $G = (V, E, w)$: We now allow $w_{xx} > 0$, and even $w_{xx} > W(x)$, for every $x \in V$.

\begin{corollary}
\label{loop}
Let $G$ be a graph with loops and $a, b : V(G) \mapsto \mathbb{R}_{\geq 0}$ two functions. Assume that $d_G(x) \geq a(x) + b(x) + 2W_G(x) - 2w_{xx}$ for every vertex $x \in V(G)$. Then there is a partition $(A,B)$ of $V(G)$ such that
\begin{enumerate}[(1)]
\item $d_A(x) \geq a(x)$ for every $x \in A$, and
\item $d_B(x) \geq b(x)$ for every $x \in B$.
\end{enumerate}
\end{corollary}

This follows by applying Theorem \ref{weighted} on a graph $G'$ derived from $G$ by omitting all loops.

We provide an example application for our result, whose solution was the motivation for this paper:

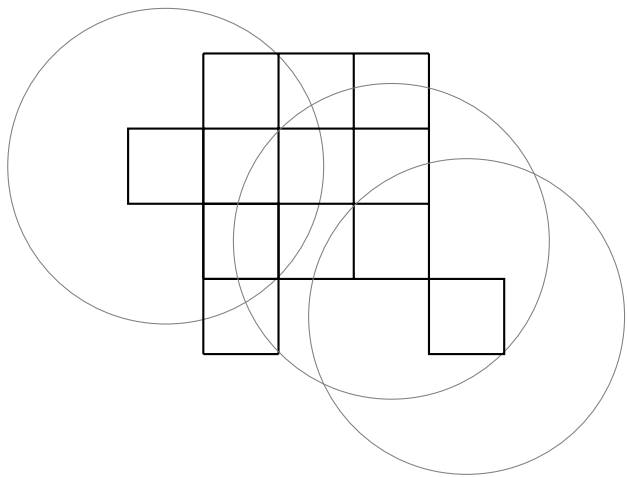
\begin{figure}[tbp]
\centering
\begin{tikzpicture}
\draw [black, thick] (-1,3) rectangle (0,2);
\draw [gray, thin] (-0.5,2.5) circle [radius=2.1];
\draw [black, thick] (0,1) grid (3,4);
\draw [black, thick] (0,0) grid (1,2);
\draw [gray, thin] (2.5,1.5) circle [radius=2.1];
\draw [black, thick] (3,0) rectangle (4,1);
\draw [gray, thin] (3.5,0.5) circle [radius=2.1];
\end{tikzpicture}
\caption[Partitioning squares]{Partitioning squares: 2-color squares so that every circle (radius=2.1) centered on a square center has most of its colored area colored as its center}
\label{squares}
\end{figure}

Let $V$ be a set of grid squares of side $1$ in $\mathbb{R}^2$, and fix a radius $r > 0$ (see Figure \ref{squares}).
Is there a non-trivial partition $(A,B)$ of $V$ such that:
\begin{enumerate}[(1)]
\item For each $x \in A$, a circle of radius $r$ drawn around its centre covers at least as much area in $A$ as in $B$, and
\item For each $x \in B$, a circle of radius $r$ drawn around its centre covers at least as much area in $B$ as in $A$
\end{enumerate}

\noindent ? (the question is easily extended to higher dimensions and to other metrics).

We are able to answer the question in the affirmative: For $r \leq \sqrt{\frac{2}{\pi}}$ square $x$ covers the majority of the radius-$r$ circle drawn around its centre, so any partition satisfies the requirements. So assume $r > \sqrt{\frac{2}{\pi}}$. We build the weighted graph $G(V,E)$ with the square set $V$ serving as the vertex set. The weight of an edge from square $x$ to square $y$, $w_{xy}$, is the area of the part of $y$ whose distance from $x$'s centre is at most $r$. Clearly $w_{xy} = w_{yx}$ and $w_{xy} \leq 1$ for every $x, y \in V$. Also, since $r > \sqrt{\frac{2}{\pi}} > \sqrt{\frac{1}{2}}$, $w_{xx} = 1$ for every $x \in V$. Therefore setting $a(x) = b(x) = d_G(x)/2$ for every $x \in V$, the existence of the sought partition follows from Corollary \ref{loop}.

\end{document}